\documentclass[a4paper,12pt]{article}

\usepackage{amsmath,amssymb,graphics,epsfig,color,cite}
\usepackage{mathrsfs,bbm}
\usepackage{amsthm} %beginproof

\usepackage{bm}

\usepackage{tikz}

\usepackage{hyperref}
 
\usepackage{geometry}
\definecolor{light}{gray}{.10}

\DeclareMathAlphabet{\mathpzc}{OT1}{pzc}{m}{it}

\newcommand{\mbf}[1]{\mathbf{#1}}

\newcommand{\ve}{\varepsilon}

\DeclareMathOperator*{\argmin}{arg\,min}

\def\Hess{{\rm Hess\,}}% Hessian

\def\supp{\mathop{\rm supp} \nolimits} % Support
\def\and {{\rm \; and \;}}

\newcommand {\pa}{\partial}

\newtheorem{theorem}{Theorem}
\newtheorem{proposition}{Proposition}

\newtheorem{lemma}{Lemma}
\newtheorem{corollary}{Corollary}
\newtheorem{remark}{Remark}

 \usepackage{titlesec}

\titleformat*{\section}{\normalsize\bfseries}
\titleformat*{\subsection}{\normalsize\bfseries}
\titleformat*{\subsubsection}{\normalsize\bfseries}
\titleformat*{\paragraph}{\normalsize\bfseries}
\titleformat*{\subparagraph}{\normalsize\bfseries}

\date{}

\title{Sharp estimate of the mean exit time of a bounded domain in the zero white noise limit}
\author{Boris Nectoux \thanks{CERMICS, \'Ecole des Ponts, Universit\'e  Paris-Est, INRIA, 77455 Champs-sur-Marne, France. E-mail: boris.nectoux@enpc.fr} }

\begin{document} 
\maketitle
\begin{abstract}
We prove  a sharp asymptotic formula for the mean exit time from a bounded domain $D\subset \mathbb R^d$ for the overdamped Langevin dynamics $$d X_t = -\nabla f(X_t) d t + \sqrt{2\ve} \ d B_t$$ 
when $\ve \to 0$ and in the case when $D$ contains a unique non degenerate minimum of $f$ and $\pa_{\mbf n}f>0$ on $\pa D$.  
This  formula    was actually first derived  in~\cite{matkowsky-schuss-77}  using formal computations and we thus   provide,  in the reversible case, the   first   proof  of it. 
As a direct consequence, we obtain when $\ve \to 0$, a sharp asymptotic estimate of the smallest eigenvalue of  the operator 
$$L_{\ve}=-\ve \Delta +\nabla f\cdot \nabla$$
associated with Dirichlet boundary conditions on $\pa D$. The approach does not require $f|_{\partial
D}$ to be a Morse function.   
The proof is based on results from~\cite{Day2,Day4} and a formula for the mean exit time from $D$  introduced in~\cite{BEGK, BGK}. 
\end{abstract}
%\tableofcontents

\section{Setting and main results}

Let us consider $(X_t)_{t\ge0}$ the stochastic process solution to the  overdamped Langevin dynamics in~$\mathbb R^d$:
\begin{equation}\label{eq.langevin}
d X_t = -\nabla f(X_t) d t + \sqrt{2\ve} \ d B_t,
\end{equation}
where $f\in C^{\infty}(\mathbb R^d,\mathbb R)$ is the potential function, $\ve>0$ is the temperature and $(B_t)_{t\geq 0}$ is a standard $d$-dimensional Brownian motion. The overdamped Langevin dynamics can be used  for instance to describe the motion of the atoms of a molecule or the diffusion of impurities in a crystal (see for instance~\cite[Sections 2 and 3]{MaSc} or \cite{chandrasekhar1943stochastic}). One of the major issues when trying to have access to the macroscopic evolution of the system  from simulations made at the microscopic level is that the process~\eqref{eq.langevin} is metastable: it is trapped during long periods of time in some regions of the configuration space. This implies that it typically  reaches a local equilibrium  of these regions long before escaping from  them.  These regions are called metastable regions (see~\cite[Chapter 8]{bovier2016metastability})  and the  move from one metastable region  to another is typically associated with a macroscopic change of  configuration of the system. The average time it takes for the process~\eqref{eq.langevin} to leave a metastable region is given by the Eyring-Kramers formula (see~\cite{HaTaBo}). 
%In practice, because of metastability,
%such   transitions can not be observed by simulating the trajectories directly  with~\eqref{eq.langevin}.  To overcome this problem, 
% several  algorithms have been designed to accelerate the exit event from a metastable region (see for instance~\cite{lelievre2016partial} for a review of these methods). 
In this work, we would like to prove, in a typical geometric setting (see \textbf{[H-D]} below),  that  the average time it takes for the process~\eqref{eq.langevin} to leave a metastable region satisfies in the small temperature regime ($\ve \to 0$) a kind of Eyring-Kramers formula even in the degenerate case   when $\argmin_{\pa D}f$ does not   consists of a finite number of non degenerate critical points of $f|_{\pa D}$.

To this end, let  us consider a $C^{\infty}$ bounded open set $D \subset \mathbb R^d$ and introduce
\begin{equation}\label{eq.tauD}
\tau_{D^c} =\inf \{ t\geq 0 | X_t \in  D^c     \}
\end{equation}
where $D^c=\mathbb R^d\setminus D$, 
the first exit time from $D$. The framework we consider in this work is the following:

\begin{itemize}
\item[] Assumption \textbf{[H-D]}: $D\subset \mathbb R^d$ is a $C^{\infty}$ bounded open set and $f\in C^{\infty}(\mathbb R^d,\mathbb R)$. The function $f$ satisfies $\pa_{\mbf n}f>0$ on $\pa D$ (where $\mbf{n}$ is the unit outward normal to $\pa D$). Moreover, $f$ has a unique critical point $x_0$ in $ D$ which is non degenerate and which satisfies $f(x_0)=\min_{\overline D} f$.
\end{itemize}

\noindent
Under the assumption \textbf{[H-D]}, it is proved in \cite[Theorem 4.1]{FrWe} (see also~\cite{sugiura1996}) that for any $x\in D$:
$$\lim_{\ve \to 0}\ve \log \mathbb E_x[ \tau_{D^c}]=\min_{\pa D}f-f(x_0).$$
 In this paper, under the assumption \textbf{[H-D]}, we prove a sharp asymptotic formula on the mean exit time from $D$ in the limit $\ve \to 0$,  a  formula  which was first obtained  using formal computations in~\cite{matkowsky-schuss-77}.   
 We also refer to~\cite{MaSc,maier1997limiting,matkowsky1981eigenvalues,matkowsky1982singular,schusstheory,schuss2009theory} where asymptotic formulas for mean exit times when $\ve \to 0$ are derived through formal computations when different geometric settings or  other diffusion processes are considered. Sharp asymptotic estimates when $\ve \to 0$ of~$\mathbb E_x[ \tau_{D^c}]$ have been obtained in~\cite[Section 4]{sugiura2001}, but these results do not apply in the setting we consider since under~\textbf{[H-D]}, $\overline{\{f< \min_{\pa D}f\}}\cap \pa D\neq \emptyset$. 
 Let us mention \cite{Ber} for a   review of the different techniques used to obtain asymptotic estimates on the mean exit time from a domain in the limit $\ve \to 0$  in various geometric settings and for an extension of the Eyring-Kramers formulas in some degenerate cases when $D=\mathbb R^d$.   Our main result is the following.
\begin{theorem}\label{th.main}
Let us assume that the assumption \textbf{[H-D]} holds.  Then, for any compact set $K\subset D$, it holds in the limit $\ve \to 0$ and  uniformly with respect  to  $x\in K$:
$$\mathbb E_x[ \tau_{D^c}]=\frac{ (2\pi\ve)^{\frac d2}   }{  \sqrt{{\rm det}\, \Hess f(x_0) }  \displaystyle\int_{\pa D}\pa_{\mbf n}f(\sigma) e^{-\frac{1}{\ve}f(\sigma)} d\sigma  }\, e^{-\frac{1}{\ve}f(x_0)} \big  (1+O(\ve) \big ),$$
where $d\sigma$ is the Lebesgue measure on $\pa D$. 
\end{theorem}
\begin{remark}
Under some assumption on $f|_{\pa D}$, an asymptotic estimate of the term $\displaystyle \int_{\pa D}\pa_{\mbf n}f(\sigma) e^{-\frac{1}{\ve}f(\sigma)} d\sigma$ in the limit $\ve \to 0$ can be obtained with Laplace's method. Two exemples are provided in~\eqref{eq.lap11} and~\eqref{eq.lap12} below. 
\end{remark}
\begin{remark}
The proof of Theorem~\ref{th.main} does not allow to obtain  a full asymptotic expansion in $\ve$ of  the remainder term $O(\ve)$. However, we expect   this asymptotic expansion to hold. 
\end{remark}
\noindent
As a consequence of
Theorem~\ref{th.main}, one obtains
an estimate in the limit $\ve \to 0$ on the first
eigenvalue of the  infinitesimal generator of the diffusion~\eqref{eq.langevin}
\begin{equation}\label{GL}
L_{\ve}=-\ve \Delta +\nabla f\cdot \nabla.
\end{equation}
with homogeneous Dirichlet boundary conditions on $\pa D$.
Let us recall that since $D\subset \mathbb R^d$ is a $C^{\infty}$ bounded open set and $f\in C^{\infty}(\mathbb R^d,\mathbb R)$,  the operator 
$L_{\ve}$ 
 with domain $H^2(D)\cap H^1_0(D)$ on $L^2(D, e^{-\frac{f(x)}{\ve}}dx)$ is self-adjoint, positive and has compact resolvent, where 
  $L^2(D, e^{-\frac{f(x)}{\ve}}dx)$ is the completion of the space $ C^{\infty}(\overline{D})$ for the norm 
$$\phi \in  C^{\infty}(\overline{D})\mapsto \int_{D} \vert \phi\vert^2e^{-\frac{1}{\ve} f}.$$
 Its smallest eigenvalue is denoted by $\lambda_{\ve}>0$.  Theorem~\ref{th.main} together with \cite[Corollary 1]{Day2} (which is recalled in Section \textbf{2.2} below) imply the following estimates on $\lambda_{\ve}$.

\begin{corollary}\label{co.lambda}
Let us assume that the assumption \textbf{[H-D]} holds. Then,   in the limit $\ve \to 0$:
$$\lambda_{\ve}=\frac{ \sqrt{{\rm det}\,\, \Hess f(x_0) } \displaystyle \int_{\pa D}\pa_{\mbf n}f(\sigma) e^{-\frac{1}{\ve}f(\sigma)} d\sigma  }{   (2\pi\ve)^{\frac d2} }\, e^{\frac{1}{\ve}f(x_0)} \big  (1+O(\ve) \big ).$$
\end{corollary}
Let us mention that sharp estimates of the smallest eigenvalues of $L_{\ve}$ have been obtained in \cite{HeNi1, LeNi, DLLN, DLLN1} in the Dirichlet case and in  \cite{Lep} in the Neumann case  when  $f|_{\pa D}$ is a Morse function (i.e. when all the critical points of $f|_{\pa D}$ are non degenerate).  When $ D=\mathbb R^d$, we refer to~\cite{BEGK,BGK,HKN,michel2017small,landim2017dirichlet}. Corollary~\ref{co.lambda} gives a general formula on the asymptotic estimate of $\lambda_{\ve}$ which allows in particular, under the assumption \textbf{[H-D]},  to deal with the case when $f|_{\pa D}$ is not a Morse function. For example,  direct consequences of Theorem~\ref{th.main} are the following: 
\begin{itemize}
\item Let us assume that $f$ is constant on $\pa D$: $f(z)\equiv f_1$ for all $z\in \pa D$. Then, for any compact set $K\subset D$, it holds:
\begin{equation}\label{eq.lap11}
\mathbb E_x[ \tau_{D^c}]=\frac{ (2\pi\ve)^{\frac d2}   }{  \sqrt{{\rm det}\,\Hess f(x_0) }  \displaystyle\int_{\pa D}\pa_{\mbf n}f(\sigma)  d\sigma  }\, e^{\frac{1}{\ve}(f_1-f(x_0))} \big  (1+O(\ve) \big ),
\end{equation}
in the limit $\ve \to 0$ and uniformly with respect to  $x \in  K$. Moreover, one has in the limit $\ve \to 0$  
$$\lambda_\ve=\frac{   \sqrt{{\rm det}\,\Hess f(x_0) } \displaystyle \int_{\pa D}\pa_{\mbf n}f(\sigma)  d\sigma  }{ (2\pi\ve)^{\frac d2}  }\, e^{-\frac{1}{\ve}(f_1-f(x_0))} \big  (1+O(\ve) \big ).$$
\item Let us assume that there exists $k\in \mathbb N^*$ such that   $\argmin_{\pa D} f=\{z_1,...,z_k\}$ and for all $j\in \{1,...,k\}$, $z_j$ is a non degenerate critical point of $f|_{\pa D}$. Then, for any compact set $K\subset D$, it holds:
\begin{equation}\label{eq.lap12}
\mathbb E_x[ \tau_{D^c}]=\sqrt{2\pi \ve}\, \sum_{j=1}^{k} \frac{\sqrt{{\rm det}\,\Hess f|_{\pa D}(z_j)} }{ \pa_{\mbf n}f(z_j) \sqrt{{\rm det}\,\Hess f(x_0) }} \, e^{\frac{1}{\ve}(f(z_1)-f(x_0))} \big  (1+O(\ve) \big )
\end{equation}
in the limit $\ve \to 0$ and  uniformly with respect to  $x \in  K$. Moreover, one has in the limit $\ve \to 0$
$$\lambda_{\ve}=\frac{1}{\sqrt{2\pi \ve}}\, \sum_{j=1}^{k} \frac{\pa_{\mbf n}f(z_j) \sqrt{{\rm det}\,\Hess f(x_0) }  }{ \sqrt{{\rm det}\,\Hess f|_{\pa D}(z_j)}} \, e^{-\frac{1}{\ve}(f(z_1)-f(x_0))} \big  (1+O(\ve) \big ).$$
In particular,  if $f|_{\pa D}$ is a Morse function, one recovers the results of \cite{HeNi1, DLLN, DLLN1} on the first eigenvalue
$\lambda_\ve$.
\end{itemize}

\section{Change of coordinates in a neighborhood of $\pa D$}

In this section, one constructs coordinates which will be useful for the computations in Section~\ref{sec.4}. The construction of these coordinates     heavily depends on the assumption $\pa_{\mbf n}f>0$ on~$\pa D$.

In all this section, we assume that the assumption \textbf{[H-D]} is satisfied.

\subsection{Eikonal solution near $\pa D$}
 
\label{sec.21} 

Let us start with the following lemma. 
\begin{lemma} \label{eikonal-boundary} Let us assume that the assumption  \textbf{[H-D]} holds.
Then, there exists a neighborhood of $\partial D$ in $\overline D$, denoted by $V_{\partial D}$, such that there exists $\Phi \in C^{\infty}(V_{\partial D},\mathbb R)$ satisfying 
\begin{equation} \label{eikonalequation-boundary}
\left\{
\begin{aligned}
\vert  \nabla \Phi \vert^2 &=  \vert  \nabla f \vert^2   \ {\rm in \ }  D \cap V_{\partial D}  \\ 
\Phi &= f \ {\rm on \ } \partial D  \\
\pa_{\mbf n} \Phi&=-\pa_{\mbf n} f \ {\rm on \ } \partial D .\end{aligned}
\right.
\end{equation}
Moreover, one has the following uniqueness results: if $\tilde \Phi$ is
a $C^{\infty}$ real valued function defined on a neighborhood
$\tilde{V}$ of $\partial D$ satisfying~\eqref{eikonalequation-boundary}, then $\tilde \Phi =\Phi$ on
$\tilde{V} \cap V_{\partial D}$. Finally, $V_{\pa D}$ can be chosen such that $\Phi>f$ on $V_{\pa D}\setminus \pa D$ and $\nabla (\Phi-f)\neq 0$ on $V_{\pa D}$.
\end{lemma}
\begin{proof} 
Let $z\in \partial D$. Using \cite[Theorem 1.5]{DiSj} or \cite[Section 3.2]{Eva} and thanks to the fact that $\pa_{\mbf n}f>0$ on $\pa D$, there exists a neighborhood of $z$ in $\overline D$, denoted by $\mathcal V_z$, such that there exists $\Phi \in C^{\infty}(\mathcal V_z,\mathbb R)$ satisfying 
$$
\left\{
\begin{aligned}
\vert  \nabla \Phi \vert^2 &=  \vert  \nabla f \vert^2   \ {\rm in \ }  D \cap \mathcal V_z  \\ 
\Phi &= f \ {\rm on \ } \partial D \cap \mathcal V_z  \\
\pa_{\mbf n} \Phi&=-\pa_{\mbf n} f \ {\rm on \ } \partial D \cap \mathcal V_z. \end{aligned}
\right.$$
Moreover, $\mathcal V_z$ can be chosen such that the following
uniqueness result holds: if a function $\tilde \Phi \in C^{\infty}(\mathcal V_z,\mathbb R)$
satisfies the previous equalities, then $\tilde \Phi =\Phi$ on $\mathcal V_z$. Now, one concludes using the fact that $\partial D$ is compact and can thus it can be covered by a finite number of these neighborhoods $(\mathcal V_z)_{z\in \partial D}$. Finally, since $\pa_{\mbf n}(\Phi-f)=-2\pa_{\mbf n} f<0$ on $\pa D$, $V_{\pa D}$ can be chosen such that $\Phi>f$ on $V_{\pa D}\setminus \pa D$ and $\nabla (\Phi-f)\neq 0$ on $V_{\pa D}$.\end{proof} 
 
\subsection{Definition of the coordinate $x_d$}

 In this section, one defines coordinates near $\pa D$ which will be convenient in the upcoming computations in Section \textbf{3}. 
Let us now consider  $\Phi$ the solution to~\eqref{eikonalequation-boundary} on the neighborhood $V_{\pa D}$ of $\pa D$ as introduced in Lemma~\ref{eikonal-boundary}. Let us define on $V_{\pa D}$:
\begin{equation}\label{fmoinsfplus}
f_+=\frac{f+\Phi}{2} \text{ and } f_-=\frac{\Phi-f}{2}.
\end{equation}
Using Lemma~\ref{eikonal-boundary}, it holds on $V_{\pa D}\setminus \pa D$:  $f_->0$ and one has on $V_{\pa D}$:
\begin{equation}\label{fmoinsfplus2}
 \nabla f_-\cdot \nabla f_+=0.
\end{equation}
Let us now consider $\delta >0$ such that 
\begin{equation*} 
V_\delta:=\{x \in \overline D, \, 0 \le f_-(x) \le \delta \}
\subset V_{\partial D}.
\end{equation*}
For any $x \in V_\delta$, the dynamics
\begin{equation}\label{eq:flow_x'}
\left\{
\begin{aligned}
\gamma_{x}'(t)&=-\frac{\nabla f_-}{|\nabla f_-|^2}(\gamma_{x}(t))\\
\gamma_x(0)&=x
\end{aligned}
\right.
\end{equation}
is well defined (since from Lemma~\ref{eikonal-boundary}, one has on $V_{\pa D}$,  $\nabla f_-\neq 0$) and is such that $\gamma_x(t_x)\in \partial D$, where $t_x=\inf\{t, \, \gamma_x(t) \in \pa D\}$. 
This is indeed a consequence of the fact that $\frac{d}{dt}
f_-(\gamma_x(t)) =-1 < 0$ on $[0,t_x)$. 
\begin{proposition}\label{de.phi}
The application
$$\Theta:\left\{ 
\begin{aligned}
V_\delta & \to  \partial D \times [0,\delta]\\
x & \mapsto (\gamma_x(t_x),t_x)
\end{aligned}
\right.
$$
defines a $C^\infty$ diffeomorphism.  The inverse application of $\Theta$    is
$$\Psi: (z,x_d) \in \partial D \times [0,\delta] \mapsto \gamma_{z}(-x_d).$$
\end{proposition}
\begin{remark}
Let us mention that the application $\Psi$ has been introduced locally in~\cite{HeNi1} and have also been used in~\cite{DLLN}. 
\end{remark}
Let us now give some properties of the  function $\Psi$ which are used in the sequel. 
Using the fact that $\Psi(z,x_d)=\gamma_{z}(-x_d)$, one obtains that for all $z\in \pa D$ and $x_d\in [0,\delta]$:
\begin{equation}\label{eq.oe1}
\nabla_{x_d}\Psi(z,x_d)=\frac{d}{dx_d}\gamma_{z}(-x_d)=\frac{\nabla f_-(z,x_d)}{\vert \nabla f_-(z,x_d)\vert^2}.
\end{equation}
Thus, one has for all $z\in \pa D$:
\begin{equation}\label{eq.oe}
\nabla_{x_d}\Psi(z,0)=-\frac{1}{\pa_{\mbf n}f(z,0)}\mbf{n},
\end{equation}
where $\mbf{n}$ is the unit outward normal to $\pa D$.  Moreover, using  the fact that $\Psi(z,0)=(z,0)$ for all $z\in \pa D$ and $\mbf{n}=-\frac{\nabla x_d}{\vert \nabla x_d\vert}$ together with \eqref{eq.oe}, it holds for all $u\in T_{z}\pa D$ and for all $v\in \mathbb R$:
\begin{equation}\label{eq.ii}
d \Psi_{(z,0)}(u+v\mbf{n})=u+\frac{v}{\pa_{\mbf n}f(z,0)}\mbf{n},
\end{equation}
and thus:
\begin{equation}\label{eq.ii2}
{\rm jac}\, \Psi(z,0)=\frac{1}{\pa_{\mbf n}f(z,0)},
\end{equation}
where ${\rm jac}\, \Psi$ is the determinant of the jacobian matrix of $\Psi$. 
Finally,  by construction (since $\frac{d}{dt}
f_-(\gamma_x(t)) =-1$) $x_d(x)=f_-(x)$ and one has $\{x_{d}=0\} = \partial
D$, $\{x_{d}>0\} = D \cap V_\delta$ and
\begin{equation}\label{Cdelta}
V_\delta =\big \{x=\Psi(z,x_d) \in \overline D, \, 0 \le x_d \le \delta \big\}.
\end{equation}
A schematic representation of $V_\delta $ is given in Figure~\ref{fig:OmegaC}.

\subsection{ Metric associated with the change of variable $x=\Psi(z,x_d)$}
 \label{sec.coo}
Let us consider $(\rho_k)_{ k\in \{1,...,N\} }\in C^{\infty}(\pa D,[0,1])^N$ a partition of unity of $\pa D$:
 \begin{equation}\label{partition}
 \text{for all } y\in \pa D, \  \sum_{k=1}^N \rho_k(y)=1
\end{equation}
 such that for all $k\in \{1,...,N\}$,  there exist smooth coordinates $x'\in \mathbb R^{d-1}$ defined by a $C^\infty$ mapping
 \begin{equation}\label{1}
\Gamma_{k}: \left\{ 
\begin{aligned}
\supp \rho_k  &\to \mathbb R^{d-1}\\
z & \mapsto x'
\end{aligned}.
\right.
\end{equation}
The  coordinates $x'\in \Gamma_k(\supp \rho_k)$ are then extended
in a neighborhood of $\supp \rho_k$ in~$ D$, as constant along the integral
curves of $\gamma'(t)=\frac{\nabla f_-}{|\nabla f_-|^2}(\gamma(t))$, for
$t \in [0,\delta]$. 
The function $
 x\mapsto (x',x_d)
$
 (where, we recall, $x_d(x)=f_-(x)$) thus defines a smooth system
of coordinates in a neighborhood $V_k$ of~$\supp \rho_k$ in~$\overline D$. Let us define 
\begin{equation}\label{chg}
 \text{ for all } (x',x_d) \in \Gamma_k(\supp \rho_k)\times  [0,\delta], \ \ \Upsilon_k(x',x_d):= \Psi\big ( \Gamma_k^{-1}(x'),x_d\big )
\end{equation}
where $\Psi$ is introduced in Proposition~\ref{de.phi}. Notice that it holds for all $(x',x_d) \in \Gamma_k(\supp \rho_k)\times  [0,\delta]$,
\begin{equation}\label{chg2}
  {\rm  Jac}\, \Upsilon_k(x',x_d):= {\rm  Jac}\, \Psi\big ( \Gamma_k^{-1}(x'),x_d\big ) \begin{pmatrix}
{\rm  Jac} \, \Gamma_k^{-1}(x') & 0 \\
0 & 1
\end{pmatrix}.
\end{equation}
where ${\rm  Jac}\, \Upsilon_k$  is the jacobian matrix of~$\Upsilon_k$.  
 In this system of coordinates,
the metric tensor $G_k(x',x_d)=\, ^t{\rm  Jac}\, \Upsilon_k(x',x_d) \, {\rm  Jac}\, \Upsilon_k(x',x_d) $ writes:
\begin{align}\label{G}
G_k:(x',x_d)\in \Gamma_k(\supp \rho_k)\times  [0,\delta]\mapsto \begin{pmatrix}
\tilde G_k(x',x_d) & 0 \\
0 & (G_k)_{dd}(x',x_d)
\end{pmatrix}
\end{align}
where $\tilde G_k$ is a $C^\infty$ square matrix of size $d-1$ and $(G_k)_{dd}$ is a $C^\infty$ positive function. Let us prove~\eqref{G}. Let us denote by $x'=(x_1',...,x_{d-1}')$. Since by construction, for all $(x',x_d) \in \Gamma_k(\supp \rho_k)\times  [0,\delta]$,   $f_-(\Upsilon_k(x',x_d))=x_d$, one has:
\begin{equation}\label{G-1}
\forall  j\in \{1,...,d-1\},\ \nabla_{x_j'} \Upsilon_k(x',x_d)\cdot \nabla f_-(\Upsilon_k(x',x_d))=0.
\end{equation}
Moreover, from~\eqref{chg2} and~\eqref{eq.oe1}, one has for all $(x',x_d) \in \Gamma_k(\supp \rho_k)\times  [0,\delta]$,
\begin{equation}\label{G-2}
\nabla_{x_d} \Upsilon_k(x',x_d)=\frac{\nabla f_-(\Upsilon_k(x',x_d))}{\vert \nabla f_-(\Upsilon_k(x',x_d))\vert^2}.
\end{equation} 
Then,  from~\eqref{G-1} and~\eqref{G-2}, it holds
$$\forall  j\in \{1,...,d-1\},\ (G_k)_{j,d}=\nabla_{x_j'} \Upsilon_k\cdot \nabla_{x_d} \Upsilon_k=0.$$
This proves~\eqref{G}. 
 Furthermore,  from~\eqref{eq.ii} and~\eqref{chg2}, one has:
\begin{equation}\label{eq.Gdd}
\text{for all } \, x' \in \Gamma_k(\supp \rho_k),\ (G_k)_{dd}(x',0)=\frac{1}{\pa_{\mbf n}f(x',0)^2}.
\end{equation}
Finally, a consequence of~\eqref{fmoinsfplus2} is that $\frac{d}{dt}
f_{+}(\gamma_x(t))=0$, where $\gamma_x$
satisfies~\eqref{eq:flow_x'} and thus, in the system of coordinates $(x',x_d)$,
the functions $f_{+}$ and~$f$ write:
\begin{align}\label{ff}
f_{+}(x',x_d)=f_{+}(x',0) \text{ and } f(x',x_d)=f_{+}(x',0)-x_d,
\end{align}
where with a slight abuse of notation, one denotes $f(\Upsilon_k(x',x_d))$ (resp. $f_+(\Upsilon_k(x',x_d))$)  by $f(x',x_d)$ (resp. by $f_+(x',x_d)$).

\section{Potential theory  and mean exit time of $D$}

\subsection{Potential theory}

Let us recall the main results from Potential theory which are used in this work. These results can be found  for instance in~\cite{bovier2016metastability}. Let us denote by $C=B(x_0,r_0)\subset D$ a closed  ball centred at $x_0$ and of radius $r_0>0$ chosen such that $B(x_0,r_0)\cap V_{\delta}=\emptyset$ where $V_{\delta}$ is given by~\eqref{Cdelta} (see Figure~\ref{fig:OmegaC}). Let~$h_{C, D^c}$  be 
 the unique weak solution in $H^1(\mathbb R^d)$ of the elliptic boundary value problem
$$
\left\{
\begin{aligned}
L_\ve \, v  &=  0    \ {\rm on \ }  D\setminus C  \\ 
v&= 0 \ {\rm on \ }   D^c \\
 v&= 1 \ {\rm on \ }  C,\\
\end{aligned}
\right.
$$
The function $h_{C, D^c}$ is called the \textit{equilibrium potential} of the \textit{capacitor} $(C, D^c)$ (as denoted in \cite[Section 2]{BEGK}). 
From elliptic regularity estimates  (see for instance~\cite[Theorem 5, Section 6.3]{Eva}), the function $h_{C, D^c}$ belongs to $C^{\infty}( \overline{D\setminus C})$. Therefore, it holds 
$$h_{C, D^c}\in H^1(D) \cap C^{\infty}( \overline{D\setminus C}).$$
Using the Dynkin's formula (see for instance \cite[Theorem 11.2]{karlin1981second}), one has for all $x\in 
\overline D$, 
\begin{equation}\label{pxc}
h_{C, D^c}(x)=\mathbb P_x[\tau_C<\tau_{D^c}],
\end{equation}
where $\tau_C =\inf \{ t\geq 0 | X_t \in C     \}$ and $\tau_{D^c}$ is defined by~\eqref{eq.tauD}.  Let us denote by $G_D$  be the Green function of $L_\ve$ associated with homogeneous Dirichlet boundary conditions on~$\pa D$. 
The   \textit{equilibrium measure} $e_{C, D^c}$  associated with $(C, D^c)$  (see \cite[Section 2]{BEGK} and more precisely the equation (2.10) there) is defined as the unique measure on $\pa C$ such that 
$$h_{C, D^c}(x)=\int_{\pa C} G_D(x,y)e_{C, D^c}(dy).$$
From \cite[Section 2]{BEGK} (see equation (2.27) there), one has the following relation:
\begin{equation}\label{rela}
\int_{\pa C} \mathbb E_z[\tau_{D^c}] \, e^{-\frac{1}{\ve}f(z)} \, e_{C, D^c}(dz) =\int_{D}e^{-\frac{1}{\ve}f(x)} \, h_{C, D^c}(x) \, dx.
\end{equation}
  Let us now define, as in \cite[Section 2]{BEGK} (see equation (2.13) there), the \textit{capacity} associated with $(C, D^c)$:
\begin{equation}\label{Cap}
{\rm cap}_C(D^c)=\int_{\pa C} e^{-\frac{1}{\ve}f(z)} e_{C, D^c}(dz).
\end{equation}

%%%%%

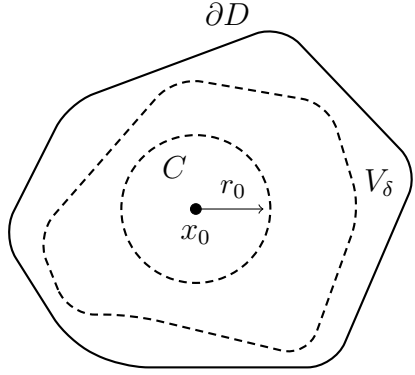
\begin{figure}[h!]
\begin{center}
\begin{tikzpicture}[scale=0.7]
\tikzstyle{vertex}=[draw,circle,fill=black,minimum size=4pt,inner sep=0pt]
\tikzstyle{ball}=[circle, dashed, minimum size=1cm, draw]
\tikzstyle{point}=[circle, fill, minimum size=.01cm, draw]
\draw [thick, rounded corners=10pt] (1,0.5) -- (-0.25,2.5) -- (1,5) -- (5,6.5) -- (7.6,3.75) -- (6,0) -- (4,0) -- (2,0) --cycle;
%%%
 
\draw [ thick,densely dashed, rounded corners=10pt] (1,1) -- (0.4,2.5) -- (3,5.5) -- (5.9,5) --(6.5,3) -- (5.5,0.2) -- (2.2,1) --cycle;

\draw [thick, densely dashed] (3.4,3) circle (1.4);
 \draw[->] (3.4,3) -- (4.67,3);
      \draw  (4.1,3.38) node[]{\small{$r_0$}};
     \draw  (6.85,3.5) node[]{$V_\delta$};
          \draw  (4,6.7) node[]{$\pa D$};
          \draw  (3,3.8) node[]{$C$};
\draw (3.4,3) node[vertex,label=south: {$x_0$}](v){};

\end{tikzpicture}

\caption{Schematic representation  in dimension $2$ of the domain $D$, $V_\delta$ (see~\eqref{Cdelta}) and of the closed  ball $C=B(x_0,r_0)$. }
 \label{fig:OmegaC}

\end{center}
\end{figure}
%%%%
 
\subsection{A first asymptotic estimate on the mean exit time of $D$}

The following results from \cite[Corollary 1]{Day2} and \cite[Theorem 2]{Day4} will be useful in the sequel. 
\begin{proposition}\label{day_res}
Let us assume that the assumption  \textbf{[H-D]} holds. Let $K\subset D$ be a compact set. Then, there exists $c>0$ such that it holds in the limit $\ve \to 0$ and uniformly with respect to  $x \in  K$:
$$\lambda_{\ve} \mathbb E_x[\tau_{D^c}]=1 +O(e^{-\frac{c}{\ve}}),$$
and
$$h_{C, D^c}(x)\ge 1-e^{-\frac{c}{\ve}},$$
where, we recall, for all $x\in \overline D$, $h_{C, D^c}(x)=\mathbb P_x[\tau_C<\tau_{D^c}]$ (see~\eqref{pxc}). 
\end{proposition}
\begin{remark}
In~\cite[Corollary 1]{Day2}, the result on $\lambda_{\ve} \mathbb E_x[\tau_{D^c}]$ is not stated with an error term. However, in view of the proof of  \cite[Corollary 1]{Day2}, the error term is $O(e^{-\frac{c}{\ve}})$ and is uniform with respect to $x$ in a compact subset of $D$.
% Let us mention that the results   stated in Proposition~\ref{day_res}  is  generalized  in \cite[Corollary 2.3, Theorem 3.6]{sugiura2001} to  the case when $D$ contains several local minima of $f$. 
\end{remark}
\noindent
Proposition~\ref{day_res} implies that in the limit $\ve \to 0$ and uniformly with respect to  $x \in  K$:
\begin{equation}\label{eq.lev}
 \mathbb E_x[\tau_{D^c}]=\mathbb E_{x_0}[\tau_{D^c}](1+O(e^{-\frac{c}{\ve}})). 
\end{equation}
We are now in position to obtain a first estimate on the mean exit time of $D$. Using~\eqref{eq.lev},~\eqref{rela} and~\eqref{Cap},  there exists $c>0$ such that in the limit $\ve \to 0$:
%\begin{equation}\label{rela1}
$$\mathbb E_{x_0}[\tau_{D^c}]=\frac{\displaystyle \int_{D}e^{-\frac{1}{\ve}f(x)} h_{C, D^c}(x) dx}{ {\rm cap}_C(D^c)  }(1 +O(e^{-\frac{c}{\ve}})).$$
%\end{equation}
Moreover, since $h_{C, D^c}\equiv 1$ on $C$, $h_{C, D^c}\le 1$ on $D$, $f(x)\ge \max_{\overline C} f>f(x_0)$ for all $x\in D\setminus C$ and using Laplace's method (since $x_0$ is non degenerate), one obtains that there exits $c>0$ such that in the limit $\ve \to 0$:
\begin{align*}
\int_{D}e^{-\frac{1}{\ve}f(x)} h_{C, D^c}(x) dx&= \int_{C}e^{-\frac{1}{\ve}f(x)} \, dx+ O(e^{-\frac{1}{\ve} (f(x_0)+c)})\\
&=\frac{ (2\pi\ve)^{\frac d2}   }{  \sqrt{{\rm det}\, \Hess f(x_0) } }\, e^{-\frac{1}{\ve}f(x_0)} \big  (1+O(\ve) \big ).
\end{align*}
Thus, one has the following result.
\begin{lemma}
Let us assume that the assumption~\textbf{[H-D]} is satisfied. Then, 
in the limit $\ve\to 0$:
\begin{equation}\label{rela1-2}
\mathbb E_{x_0}[\tau_{D^c}]=\frac{ (2\pi\ve)^{\frac d2}   }{  \sqrt{{\rm det}\, \Hess f(x_0) } \, {\rm cap}_C(D^c)  }\big (1 +O(\ve )\big),
\end{equation}
where $\tau_{D^c}$ is defined by~\eqref{eq.tauD} and ${\rm cap}_C(D^c)  $ by~\eqref{Cap}.
\end{lemma}
To prove Theorem~\ref{th.main}, it remains to give an estimate on ${\rm cap}_C(D^c) $ in the limit $\ve \to 0$. This is the purpose of the next section. 

\section{Proofs of Theorem~\ref{th.main} and Corollary~\ref{co.lambda}} 
 \label{sec.4}
In this section, one obtains  sharp lower and upper bounds on the capacity~${\rm cap}_C(D^c)$. The following proof is inspired by the one made in~\cite[Theorem 3.1]{BEGK}. However, the functions involved here to get the lower and upper bounds on the capacity~${\rm cap}_C(D^c)$ are constructed in the whole neighborhood  $V_\delta$ of $\pa D$ using the coordinates $(x',x_d)$ introduced in Section~\ref{sec.coo}.  This is indeed    needed   since the whole boundary of $D$ appears in the asymptotics estimates stated in Theorem~\ref{th.main}. Moreover, the coordinates $(x',x_d)$   are particularly convenient for computations since in these coordinates, the tensor metric   has the form~\eqref{G}. 
Finally the support of these functions  is $V_\delta$ and thus  does not depend on $\ve$. This allows us to obtain a  remainder term  $O(\ve)$ in Theorem~\ref{th.main}.   \\
From \cite[Section 2]{BEGK}, one has the following variational principle:
\begin{align}
\label{eq.mini}
{\rm cap}_C(D^c)&=\ve \int_{D\setminus C} \big \vert \nabla h_{C, D^c}(x)\big \vert^2e^{-\frac{1}{\ve}f(x)}dx\\
\nonumber
&=\inf_{h\in H_{C, D^c}}\ve \int_{D\setminus C}\big \vert \nabla h (x)\big \vert^2e^{-\frac{1}{\ve}f(x)}dx,
\end{align}
where
$$H_{C, D^c}=\big \{ h\in H^1(\mathbb R^d),\,  h(x)=1 \text{ for } x\in C, \,    h(x)=0 \text{ for } x\in  D^c\}.$$
Formula~\eqref{eq.mini}  holds since the function $h_{C, D^c}$ is   a minimizer of the functional
$$h\in H_{C, D^c} \mapsto \ve \int_{D\setminus C}\big \vert \nabla h (x)\big \vert^2e^{-\frac{1}{\ve}f(x)}dx.$$
Using this variational principle, one can get a sharp upper bound on ${\rm cap}_C(D^c)$ by choosing a suitable function $h\in H_{C, D^c}$.

\subsection{Upper bound on ${\rm cap}_C(D^c)$}

In this section, one gets  a sharp  upper bound on ${\rm cap}_C(D^c)$. 
Let $V_{\delta}$ be defined by~\eqref{Cdelta} and let $h\in H_{C, D^c}$. 
From~\ref{eq.mini}, one has 
\begin{equation}\label{up1}
{\rm cap}_C(D^c)\le \ve \int_{V_{\delta}}\big \vert \nabla h (x)\big \vert^2e^{-\frac{1}{\ve}f(x)}dx +  \ve \int_{D\setminus V_{\delta}}\big \vert \nabla h (x)\big \vert^2e^{-\frac{1}{\ve}f(x)}dx.
\end{equation}
From~\eqref{partition},~\eqref{1},~\eqref{chg} and~\eqref{G}, one has:
\begin{align}
\nonumber
&\ve \int_{V_{\delta}}\big \vert \nabla h (x)\big \vert^2e^{-\frac{1}{\ve}f(x)}dx\\
\nonumber
&= \ve \sum_{k=1}^N \ \int_{x'\in \Gamma_k(\supp \rho_k)}\!\!\!\!\rho_k(\Gamma_k^{-1}(x'))\\
\label{eq.egalite}
&\quad \times\int_0^{\delta} \,^t\tilde \nabla h(x',x_d) G_k(x',x_d)^{-1} \tilde \nabla h(x',x_d) \,  e^{-\frac{1}{\ve}f(x',x_d)}\, {\rm jac}\, \Upsilon_k(x',x_d) \, dx_d\, dx'
\end{align}
where $^t\tilde \nabla=(\pa_{x'}, \pa_{x_d})$, $\Upsilon_k$ is defined by~\eqref{chg}, $G_k$ is the tensor metric associated with the change of variable $x=\Upsilon_k(x',x_d)$ (see~\eqref{G}) and ${\rm jac}\, \Upsilon_k = \sqrt{ {\rm det} \, G_k }$ is the jacobian of~$\Upsilon_k$. \\
Let us now consider   the following function:
$$x_d\in [0,\delta]\mapsto g(x_d)= \frac{\displaystyle \int_0^{x_d} e^{-\frac{t}{\ve}} dt }{ \displaystyle \int_0^{\delta} e^{-\frac{t}{\ve}} dt }=\frac{ 1-e^{-\frac{x_d}{\ve}}}{  1-e^{-\frac{\delta}{\ve}} },$$
which satisfies $g(0)=0$ and $g(\delta)=1$. Let $h:V_\delta\to \mathbb R$ be such that  
$$  h\circ \Psi (z,x_d):=g(x_d), \text{ for all } (z,x_d)\in \pa D\times [0,\delta].$$
The function $h$ is then extended by $1$ in $D\setminus V_{\delta}$ and by $0$ outside $D$. Thus, $h$ belongs to $H_{C, D^c}$ since $C\subset D\setminus V_{\delta}$.
For all $k\in \{1,...,N\}$ and for all  $(x',x_d)\in  \Gamma_k(\supp \rho_k) \times [0,\delta]$, denoting with   a slight abuse of notation $h\circ \Upsilon_k$ by $h$, 
one has $h(x',x_d)=g(x_d)$ and then for any  $(x',x_d)\in \Gamma_k(\supp \rho_k)\times [0,\delta]$:
$$\pa_{x'}h(x',x_d)=0 \text{ and } \pa_{x_d}h(x',x_d)=\frac{d}{dx_d}g(x_d).$$
 From~\eqref{G},~\eqref{ff},~\eqref{up1}, and~\eqref{eq.egalite} together with the fact that $  \nabla h=0$ on $D\setminus V_{\delta}$, one has:
\begin{align*}
{\rm cap}_C(D^c)&\le \ve \sum_{k=1}^N\  \int_{x'\in \Gamma_k(\supp \rho_k)}\!\! \frac{ \rho_k(\Gamma_k^{-1}(x')) \, e^{-\frac{1}{\ve}f_+(x',0)} }{ \ve^2\big (1-e^{-\frac{\delta}{\ve}}\big)^2 }  \\
&\quad \times\int_0^{\delta}  e^{-\frac{x_d}{\ve}} (G_k)_{dd}(x',x_d)^{-1}\, {\rm jac}\, \Upsilon_k(x',x_d)dx_d\, dx'.
\end{align*}
Now let us notice that for any function $\varphi\in C^{\infty}(\Gamma_k(\supp \rho_k)\times [0,\delta], \mathbb R^*_+)$, one has in the limit $\ve \to 0$:
\begin{align}
\label{eq.lap}
  \int_0^{\delta} \varphi(x',x_d) e^{-\frac{x_d}{\ve}}dx_d  &=\ve\, \varphi(x',0)  \big(1   + O(\ve)\big),
\end{align}
uniformly with respect to $x'\in \Gamma_k(\supp \rho_k)$.
Thus, applying~\eqref{eq.lap} with $\varphi= (G_k)_{dd}^{-1}\, {\rm jac}\, \Upsilon_k$,  it holds in the limit $\ve \to 0$:
\begin{align*}
{\rm cap}_C(D^c)&\le   \sum_{k=1}^N \ \int_{x'\in \Gamma_k(\supp \rho_k)}  \frac{\rho_k(\Gamma_k^{-1}(x')) e^{-\frac{1}{\ve}f_+(x',0)}}{ \big (1-e^{-\frac{\delta}{\ve}}\big )^2 }\\
&\quad \times  (G_k)_{dd}(x',0)^{-1} \, {\rm jac}\,\Upsilon_k(x',0)dx'\, \big  (1+O(\ve) \big ).
\end{align*}
Finally, using~\eqref{eq.ii},~\eqref{chg2} and~\eqref{eq.Gdd}, it holds in the limit $\ve \to 0$:
\begin{align*}
{\rm cap}_C(D^c)&\le   \sum_{k=1}^N \ \int_{x'\in \Gamma_k(\supp \rho_k)}  \!\!\!\!  \rho_k(\Gamma_k^{-1}(x')) e^{-\frac{1}{\ve}f(x',0)}   \pa_{\mbf n}f(x',0) \, {\rm jac}\, \Gamma_k^{-1}(x')dx' \, \big  (1+O(\ve) \big ).
\end{align*}
Therefore, since   from \eqref{fmoinsfplus} and Lemma~\ref{eikonal-boundary}, $f(x',0)=f(x)$ for all  $x=\Upsilon_k(x',0) \in \pa D$, 
%and   using in addition~\eqref{partition},
%and the fact that ${\rm jac}\, \Gamma_k^{-1}(x')dx' =  d\, \Gamma_k^{-1}(x')$, 
one has following result.
\begin{lemma}
Let us assume that the assumption~\textbf{[H-D]} is satisfied. Then,  
it holds in the limit $\ve \to 0$:
\begin{align}\label{minor1}
{\rm cap}_C(D^c)&\le   \int_{\pa D} \pa_{\mbf n}f(\sigma)\, e^{-\frac{1}{\ve}f(\sigma)}d\sigma\,   \big  (1+O(\ve) \big ),
\end{align}
where, we recall, ${\rm cap}_C(D^c)$ is defined by~\eqref{Cap}.
\end{lemma}
Let us now give a sharp lower bound on ${\rm cap}_C(D^c)$.
 
\subsection{Lower bound on ${\rm cap}_C(D^c)$}

In this section, one gets a   sharp lower bound on ${\rm cap}_C(D^c)$. 
Let $V_{\delta}$ be defined by~\eqref{Cdelta}. Using~\eqref{eq.mini},~\eqref{partition},~\eqref{1},~\eqref{chg} and~\eqref{G},   one has:
\begin{align}
\nonumber
{\rm cap}_C(D^c)&\ge \ve \int_{V_{\delta}}\big \vert \nabla h_{C, D^c} (x)\big \vert^2e^{-\frac{1}{\ve}f(x)}dx\\
\label{eq.e1}
&\ge \ve \sum_{k=1}^N \ \int_{x'\in \Gamma_k(\supp \rho_k)}\!\!\!\!\!\!\!\!\rho_k(\Gamma_k^{-1}(x'))\, e^{-\frac{1}{\ve}f_+(x',0)}\int_0^{\delta}\, L_k(x',x_d)\, dx_d\, dx'
\end{align}
with 
$$
 L_k(x',x_d):=  \big \vert \partial_{x_d} h_{C, D^c}(x',x_d) \big \vert^2\, (G_k)_{dd}(x',x_d)^{-1}  e^{\frac{1}{\ve}x_d}\, {\rm jac}\, \Upsilon_k(x',x_d).
$$
Let us define for $k\in \{1,...,N\}$ and $(x',x_d)\in  \Gamma_k(\supp \rho_k) \times [0,\delta]$: 
\begin{equation}\label{eq.phi}
\chi_k(x',x_d):=(G_k)_{dd}(x',x_d)^{-1}\,{\rm jac}\, \Upsilon_k(x',x_d).
\end{equation}
Notice that from~\eqref{eq.ii2},~\eqref{chg2}, and~\eqref{eq.Gdd}, it holds for all $x'\in \Gamma_k(\supp \rho_k)$,
\begin{equation}\label{eq.chi_k(x',0)}
\chi_k(x',0)=\partial_{\mbf n} f(x',0) \,{\rm jac}\, \Gamma_k^{-1}(x').
\end{equation}
The function $\chi_k$ satisfies 
\begin{equation}\label{chi-minimisation}
\min_{\Gamma_k(\supp \rho_k)\times [0,\delta]}\chi_k >0.
\end{equation} 
Let us consider  $k\in \{1,...,N\}$ and $x'\in  \Gamma_k(\supp \rho_k)$. Then, it holds:
\begin{align}
\nonumber
\int_0^{\delta}\, L_k(x',x_d)\, dx_d &=
\int_0^{\delta}\, \big \vert \pa_{t} h_{C, D^c}(x',t) \big\vert^2\, \chi_k(x',t)\,e^{\frac{t}{\ve}}\,dt  \\
\label{eq.mi1}
&\ge  \inf_{\substack{ g\in H^1(0,\delta)\\ g(0)=0\\ g(\delta)= h_{C, D^c}(x',\delta)}}\int_0^{\delta}\Big \vert\frac{d}{dt} g(t)\Big\vert^2\chi_k(x',t) \,e^{\frac{t}{\ve}}dt. 
\end{align}
Let us now prove that
\begin{equation}\label{eq.mi}
\inf_{\substack{ g\in H^1(0,\delta)\\ g(0)=0\\ g(\delta)= h_{C, D^c}(x',\delta)}}\int_0^{\delta}\Big \vert\frac{d}{dt} g(t)\Big\vert^2\chi_k(x',t) \,e^{\frac{t}{\ve}}dt=\int_0^{\delta} \big \vert\pa_t g_{x'}^*(t)\big \vert^2\chi_k(x',t) \,e^{\frac{t}{\ve}}dt,
\end{equation}
where 
$$g_{x'}^*(t)=\frac{\displaystyle \int_0^{t}\chi_k(x',s)^{-1} e^{-\frac{s}{\ve}} ds    }{\displaystyle \int_0^{\delta} \chi_k(x',s)^{-1} e^{-\frac{s}{\ve}} ds }\, h_{C, D^c}(x',\delta).$$
The set 
$K=\{ g\in H^1(0,\delta), g(0)=0 \text{ and }  g(\delta)= h_{C, D^c}(x',\delta)\big\} $ is a closed convex subset of $H^1(0,\delta)$ and the  functional 
$$F: \theta \in   H^1(0,\delta) \mapsto \int_0^{\delta}  \Big \vert\frac{d}{dt} \theta(t)\Big\vert^2\chi_k(x',t) \,e^{\frac{t}{\ve}}dt$$
is continuous and from~\eqref{chi-minimisation}, it is strongly convex. Furthermore, since for all $u\in K$, $u(0)=0$, there exists $C>0$ such that for all $g\in K$, 
$$\int_0^\delta g^2\le C \int_0^\delta \Big \vert\frac{d}{dt} g(t)\Big\vert^2.$$
Thus, using in addition~\eqref{chi-minimisation},  there exists $c>0$ such that for all $g\in K$, 
\begin{equation}\label{eq.co}
\int_0^\delta g^2 +  \int_0^\delta \Big \vert\frac{d}{dt} g(t)\Big\vert^2\le c \, F(g).
\end{equation}
Let us consider a  sequence $(g_n)_{n\ge 0}\in K^{\mathbb N}$ such that $\lim_{n\to \infty}F(g_n)= \inf_K F$.  Then, from~\eqref{eq.co}, $(g_n)_{n\ge 0}$ is a bounded sequence in $H^{1}(0,\delta)$ and thus converges for the weak topology  of $H^1(0,\delta)$ towards some $g\in H^1(0,\delta)$. Since $F$ is continuous and convex on~$K$, it is a lower semi-continuous function for the weak topology in $H^1(0,\delta)$. Therefore, $ \inf_K F\le F(g)$ and since  $g\in K$, $g$ is a minimizer of $F$ on $K$. Finally, because $F$ is strongly convex, $g$ is the unique  minimizer of $F$ on $K$. Let $\alpha\in \mathbb R$ and $\varphi\in C_c^\infty(0,\delta)$. Then, it holds $g+\alpha \varphi\in K$ and thus 
$$ F(g)\le F(g+\alpha \varphi)= F(g) +2 \alpha \int_0^\delta \frac{d}{dt}  g(t)\, \frac{d}{dt}  \varphi(t) \, \chi_k(x',t) \,e^{\frac{t}{\ve}}dt + \alpha^2 \,F(\varphi) .$$
Thus, $g\in H^1(0,\delta)$ is a weak solution to the following one dimensional Dirichlet problem on $(0,\delta)$: 
\begin{equation}\label{eq.ell}
\left\{ 
\begin{aligned}
\frac{d}{dt}\Big( e^{\frac{t}{\ve}}\chi_k(x',t)\, \frac{d}{dt} g(t)\Big) &=0 \text{ on }  (0,\delta),\\
g(0)&=0, \\
g(\delta)&=h_{C, D^c}(x',\delta).
\end{aligned}
\right.
\end{equation}
From~\eqref{chi-minimisation}, one can use the Lax-Milgram Theorem which implies that there exists a unique solution in $H^1(0,\delta)$ of~\eqref{eq.ell}. Clearly, this solution is given by
$$g_{x'}^*(t)=\frac{\displaystyle \int_0^{t}\chi_k(x',s)^{-1} e^{-\frac{s}{\ve}} ds    }{\displaystyle \int_0^{\delta} \chi_k(x',s)^{-1} e^{-\frac{s}{\ve}} ds }\, h_{C, D^c}(x',\delta),$$
and thus $g=g_{x'}^*$. 
This concludes the proof of~\eqref{eq.mi}. 
Using~\eqref{eq.e1},~\eqref{eq.mi1} and~\eqref{eq.mi} together with the second statement in Proposition~\ref{day_res}, there exists $c>0$ such that in the limit $\ve \to 0$:
\begin{align*}
{\rm cap}_C(D^c)& \ge \ve   \sum_{k=1}^N \ \int_{x'\in \Gamma_k(\supp \rho_k)}\!\!\!\!\!\! \rho_k(\Gamma_k^{-1}(x'))\, e^{-\frac{1}{\ve}f_+(x',0)}\\
&\quad \times \int_0^{\delta}\, \big \vert \pa_{x_d} g_{x'}^*(x_d) \big \vert^2\,\chi_k(x',x_d) e^{\frac{1}{\ve}x_d} dx_d\, dx'\\
&= \ve   \sum_{k=1}^N\ \int_{x'\in \Gamma_k(\supp \rho_k)}   \!\!\!\!\!\rho_k(\Gamma_k^{-1}(x'))\,  h^2_{C, D^c}(x',\delta)\, e^{-\frac{1}{\ve}f_+(x',0)} \\
&\quad \times\int_0^{\delta} \frac{  \chi_k(x',x_d)^{-1}}{ \Big( \displaystyle \int_0^{\delta} \chi_k(x',s)^{-1} e^{-\frac{s}{\ve}} ds  \Big)^2  }\, e^{-\frac{1}{\ve}x_d}  \, dx_d\, dx'\\
&\ge  \ve  (1-e^{-\frac{c}{\ve}}\big )^2 \sum_{k=1}^N \ \int_{x'\in \Gamma_k(\supp \rho_k)}   \!\!\!\!\!\rho_k(\Gamma_k^{-1}(x')) \, e^{-\frac{1}{\ve}f_+(x',0)} \\
&\quad \times\int_0^{\delta} \frac{  \chi_k(x',x_d)^{-1}}{ \Big( \displaystyle \int_0^{\delta} \chi_k(x',s)^{-1} e^{-\frac{s}{\ve}} ds  \Big)^2  }\, e^{-\frac{1}{\ve}x_d}  \, dx_d\, dx'.
\end{align*} 
Then, using~\eqref{eq.lap},~\eqref{eq.phi}, and~\eqref{eq.chi_k(x',0)}, one has  in the limit $\ve \to 0$:
\begin{align*}
{\rm cap}_C(D^c)&\ge   \sum_{k=1}^N \int_{x'\in \Gamma_k(\supp \rho_k)}  \!\!\!\!  \rho_k(\Gamma_k^{-1}(x')) e^{-\frac{1}{\ve}f(x',0)}   \pa_{\mbf n}f(x',0) \, {\rm jac}\, \Gamma_k^{-1}(x')dx' \, \big  (1+O(\ve) \big ).
\end{align*}
Therefore, since   from \eqref{fmoinsfplus} and Lemma~\ref{eikonal-boundary}, $f(x',0)=f(x)$ for all  $x=\Upsilon_k(x',0) \in \pa D$, 
%and   using in addition~\eqref{partition},
% and the fact that ${\rm jac}\, \Gamma_k^{-1}(x')dx' =  d\, \Gamma_k^{-1}(x')$,  
 one has the following lower bound on ${\rm cap}_C(D^c)$. 
\begin{lemma}
Let us assume that the assumption~\textbf{[H-D]} is satisfied. Then,  
it holds in the limit $\ve \to 0$:
\begin{align}\label{minor2}
{\rm cap}_C(D^c)&\ge   \int_{\pa D} \pa_{\mbf n}f(\sigma)\, e^{-\frac{1}{\ve}f(\sigma)}d\sigma\,   \big  (1+O(\ve) \big ).
\end{align}
where, we recall, ${\rm cap}_C(D^c)$ is defined by~\eqref{Cap}.
\end{lemma}

Theorem~\ref{th.main} is then a consequence of~\eqref{minor1} and~\eqref{minor2} together with~\eqref{rela1-2} and~\eqref{eq.lev}. Corollary~\ref{co.lambda} is a consequence of Theorem~\ref{th.main} and Proposition~\ref{day_res}.

\bigskip
\noindent
\textbf{Acknowledgements.}
This work was motivated by a question of B. Helffer. 
I am very grateful to T. Leli\`evre and D. Le Peutrec  for their
suggestions and help. 
This work is supported by the European Research Council under the European
Union's Seventh Framework Programme (FP/2007–2013)/ERC Grant Agreement
number 614492.

\begin{small}

\bibliography{lambdah} %You need to replace "rsc" on this
%line with the name of your .bib file
\bibliographystyle{plain} %the RSC's .bst file
 \end{small}

\end{document}